\documentclass[12pt]{amsproc}
\usepackage[english]{babel}
\usepackage[latin1]{inputenc}
\usepackage{amsmath}
\usepackage{amssymb}
\usepackage{color}
\input xy
\xyoption{all}

\newtheorem{theorem}{Theorem}[section]

\theoremstyle{definition}
\newtheorem{definition}[theorem]{Definition}
\newtheorem{example}[theorem]{Example}
\newtheorem{Lemma}[theorem]{Lemma}
\newtheorem{corollary}[theorem]{Corollary}
\newtheorem{Prop}[theorem]{Proposition}
\newtheorem{Thm}[theorem]{Theorem}
\theoremstyle{remark}
\newtheorem{remark}[theorem]{Remark}
\newtheorem{question}[theorem]{Question}

\numberwithin{equation}{section}
\newcommand{\Cal}[1]{{\mathcal #1}}

\newcommand{\Ann}{\operatorname{Ann}}
\newcommand{\Tot}{\operatorname{Tot}}
\newcommand{\Ker}{\operatorname{Ker}}

\newcommand{\Spec}{\operatorname{Spec}}
\newcommand{\Reg}{\operatorname{Reg}}

\newcommand{\Z}{\mathbb{Z}}
\newcommand{\Q}{\mathbb{Q}}
\newcommand{\cmat}{\left(\begin{array}}
\newcommand{\fmat}{\end{array}\right)}

\newcommand{\gm}{\mathfrak{m}}
\newcommand{\gp}{\mathfrak{p}}
\newcommand{\gq}{\mathfrak{q}}
\newcommand{\ga}{\mathfrak{a}}
\newcommand{\gb}{\mathfrak{b}}
\newcommand{\gc}{\mathfrak{c}}
\newcommand{\gi}{\mathfrak{i}}

\newcommand{\f}{\mathfrak}

\makeatletter
\@namedef{subjclassname@2020}{\textup{2020} Mathematics Subject Classification}
\makeatother

\begin{document}

 \title[Some remarks on Pr\"ufer rings with zero-divisors]{Some remarks on Pr\"ufer rings with zero-divisors}
  

\author{Federico Campanini}
\address{Dipartimento di Matematica, Universit\`a di Padova, 35121 Padova, Italy}
\email{federico.campanini@phd.unipd.it}
 
\author{Carmelo Antonio Finocchiaro}
\address{Dipartimento di Matematica e Informatica, Università degli Studi di Catania, 95125 Catania}
\email{cafinocchiaro@unict.it}
\thanks{The first and the second author were partially supported by Dipartimento di Matematica ``Tullio Levi-Civita'' of Universit\`a di Padova (Project BIRD163492/16 ``Categorical homological methods in the study of algebraic structures'' and Research program DOR1690814 ``Anelli e categorie di moduli''). The second author was partially supported by GNSAGA,  by the project ``Proprietà algebriche locali e globali di anelli associati a curve e ipersuperfici" PTR 2016-18 - Dipartimento di Matematica e Informatica - Università di Catania, by the project PIACERI ``PLGAVA - Proprietà locali e globali di anelli e di varietà algebriche" Dipartimento di Matematica e Informatica - Università di Catania, and by the research program ``Reducing complexity in algebra, logic,
	combinatorics - REDCOM'' ``Ricerca Scientifica di Eccellenza 2018'' Fondazione Cariverona. }



\begin{abstract} Let $A$ be the fiber product $R\times_TB$, where $B\to T$ is a surjective ring homomorphism with regular kernel and $R\subseteq T$ is a ring extension where $T$ is an overring of $R$. In this paper we provide a characterization of when $A$ has distinguished Pr\"ufer-like properties and new constructions  of Pr\"ufer rings with zero-divisors. Furthermore we give examples of homomorphic images of Pr\"ufer rings that are Pr\"ufer without assuming that the kernel of the surjection is regular. Finally we provide some remarks on the ideal theory of  pre-Pr\"ufer rings. 
\end{abstract}

\maketitle
\section*{Introduction}
The notion of Pr\"ufer domain, introduced in 1932 by H. Pr\"ufer \cite{Prufer}, reached during last century a relevant role and a remarkable impact in the development of Multiplicative Ideal Theory.  While the notion of Dedekind domain globalizes that of discrete valuation ring, the concept of Pr\"ufer domain arises in order to globalize that of valuation domain in the non-local and non-Noetherian context. For instance the ring 
$$
{\rm Int}(\mathbb Z):=\{f(T)\in \mathbb Q[T]\mid f(\mathbb Z)\subseteq \mathbb Z \}
$$
of integer-valued polynomials over $\mathbb Z$ and the ring of entire functions on the complex plane $\mathbb{C}$ are natural examples of non-local and non-Noetherian Pr\"ufer domains. For a deeper insight on this circle of ideas and an excellent account of the work done on Pr\"ufer domains we refer the reader to \cite{fo-hu-pa,gilmer}. 

In 1970 M. Griffin (see \cite{Gr}) provided a generalization of the notion of Pr\"ufer domain for arbitrary commutative rings (see condition (5) of Definition \ref{P-like} below). It is worth recalling that several equivalent conditions defining Pr\"ufer domains are not equivalent for rings with zero-divisors (as it was observed in \cite{Ba-Gl-JA}). Such conditions define relevant distinct  ``Pr\"ufer-like" classes of rings. We will recall them here for the reader's convenience.
\begin{definition}\label{P-like}
A ring $R$ (that from now on is always assumed to be commutative with multiplicative identity) is said to be:
\begin{enumerate}
	\item
	{\em semihereditary}, if every finitely generated ideal of $R$ is projective;
	\item
	{\em of weak global dimension $\leq 1$} ($\operatorname{w.gl.dim}(R)\leq 1$), if every ideal of $R$ is flat;
	\item
	{\em arithmetical} if the lattice of all ideals of $R$ is distributive, that is, given ideals $\ga, \gb$ and $\gc$ of $R$, the following equality  $$\ga\cap (\gb +\gc)=(\ga \cap \gb)+(\ga \cap \gc)$$ holds;
	\item
	{\em a Gaussian ring} if every polynomial $f\in R[X]$ is {\em a Gaussian polynomial},  that is, for every $g \in R[X]$ one has $$c_R(f)c_R(g)=c_R(fg)$$ (here $c_R(f)$ denotes, as usual, the content of $f$, i.e., the ideal of $R$ generated by the coefficients of $f$);
	\item
	{\em a Pr\"ufer ring} if every finitely generated regular ideal of $R$ is invertible;
\end{enumerate}
\end{definition}
For the sake of simplicity we will say that a ring \emph{has Pr\"ufer condition $(n)$} ($1\leq n\leq 5$) if it satisfies condition $(n)$ of Definition \ref{P-like}. 

Note that by \cite[Theorem 3.12]{Ba-Gl-JA} such conditions are equivalent whenever the total ring of fractions of $R$ is absolutely flat (and thus, in particular, when $R$ is a domain all of them collapse to the notion of Pr\"ufer domain). 

In \cite{Boy}, the author investigates about the behaviour of the  conditions given in Definition \ref{P-like} in {\em regular conductor squares}, that is, in pullback diagrams defined as follows: let $A\subseteq B$ be a ring extension, with $\gc:=[A:_A B]=\{a \in A \mid aB \subseteq A\}$. Assume that $\gc$ is a regular ideal of $B$ and consider the following pullback diagram
\begin{equation*}
\xymatrix{
A \ar[r] \ar@{^{(}->}[d] &  R \ar@{^{(}->}[d] \\
B \ar[r]^{\pi} & T
} 
\end{equation*}
where $R=A/\gc$, $T=B/\gc$ and the bottom orizontal arrow is the canonical projection. In \cite{Boy}, the author gives a precise description of when $A$ has Pr\"ufer condition $(n)$ for $1\leq n\leq 5$. This characterization is expressed in terms of the local features of the pullback diagram. One of the goals of this paper (see Section 3) is study the transfer of Pr\"ufer-like properties in such pullback diagrams from a different point of view, namely, we characterize when the ring $A$ of the pullback diagram has some Pr\"ufer-like properties in terms of direct conditions on the given data $R,B,T$, under the reasonable (and quite natural) assumption that $T$ is an overring of $R$. This approach allow us to generalize results of \cite{FO1} and \cite{HT}, for instance, and to provide new natural examples of Pr\"ufer rings. It is worth noting that this extra assumption is actually a necessary condition in many cases where $A$ is a Pr\"ufer ring. 
For our purposes, a useful tool  will be  the notion of \emph{Pr\"ufer extension}, recently developed by M. Knebush and D. Zhang in \cite{KZ} (see Section 2  for some results of independent interest regarding Pr\"ufer extensions whose base ring is local).

 Section 4 is devoted to study homomorphic images and Pr\"ufer rings from various perspectives. As it is well known, the class of Pr\"ufer rings is not stable under quotients, but a homomorphic image of a Pr\"ufer ring is still Pr\"ufer, provided that the kernel of the ring surjection is regular. We introduce another class of ring morphisms that preserves the Pr\"ufer property under homomorphic images.  This class will be that of  \emph{regular morphisms of rings}, that is, morphisms $A\to B$ for which the inverse image of the set of regular elements of $B$ consists of regular elements of $A$.

Finally, we investigate about the ideal theory of the so called \emph{pre-Pr\"ufer rings}, that is, rings whose non-trivial homomorphic images are Pr\"ufer rings \cite{BoiShe}. We show that any two given prime ideals of a localization of a pre-Pr\"ufer ring are comparable, provided that at least one of them is regular (this is a generalization of \cite[Theorem 1.2]{BoiShe}). Furthermore, we prove that any Noetherian pre-Pr\"ufer ring has Krull dimension $\leq 1$.

\section{Notation and preliminaries}
For a ring $A$, we  denote by $\Cal U(A)$ (resp., $Z(A)$) the set of invertible elements (resp., zero-divisors) of $A$. The total ring of fractions of $A$ is denoted by $\Tot(A)$. An \emph{overring of $A$} is a subring of $\Tot(A)$ containing $A$. An element of $A$ that is not a zero-divisor is called \emph{a regular element}. The set of regular elements of $A$ is denoted by $\Reg(A)$. An ideal of $A$ is \emph{regular} if it contains a regular element. Given a ring extension $A\subseteq B$, an intermediate ring $X$ between $A$ and $B$ and $A$-submodules $F,G$ of $B$, set 
$$
[F:_XG]:=\{x\in X\mid xG\subseteq F \} .
$$
It is easily seen that $[A:_B B]=[A:_A B]$ is the biggest ideal of $A$ that is also an ideal of $B$, and it
is called \emph{the conductor of the ring extension $A\subseteq B$}. It is well knonw that, if the conductor $[A:_B B]$ is a regular ideal of $B$, then $B$ is an overring of $A$ because, if $a\in [A:_B B]$ is a regular element in $B$, then $b\mapsto \frac{ab}{a}$ defines a ring embedding $B\hookrightarrow \Tot(A)$.

If $f:A\to C,g:B\to C$ are given morphisms of rings, we denote by
$$
f\times_Cg:=A\times_CB:=\{(a,b)\in A\times B\mid f(a)=g(b) \}
$$
the fiber product of $f$ and $g$ (in the category of rings). Constructions of this type are very useful for producing examples, for instance rings whose prime spectra have to satisfy prescribed conditions. We recall now some stardard facts that will be freely used in the remaining part of the paper and we address the reader to \cite{FO1} for more details regarding this topic. 
\begin{remark}\label{fiber-product}
Let $\pi:B\to T$ be a surjective ring homomorphism and let $R$ be a subring of $T$. Consider the subring $A:=\pi^{-1}(R)$ of $B$. 
\begin{enumerate}
	\item $\ker(\pi)$ is a common ideal of $A$ and $B$. In particular, $\ker(\pi)\subseteq [A:_A B]$.
	\item $A$ is canonically isomorphic to  the fiber product of $\pi$ and of the inclusion $ R\hookrightarrow T$). 
	\item $A$ is a local ring if and only if $R$ is local and $\ker(\pi)$ is contained in the Jacobson radical of $B$ \cite[Corollary~2.4~(3)]{DFF2}.
\end{enumerate}
\end{remark}
We now recall a useful characterization of Pr\"ufer rings due to Griffin \cite{Gr}. If $\gp$ is a prime ideal of a ring $R$, the pair $(R,\gp)$ is said to have \emph{the regular total order property} if, whenever $\ga,\gb$ are ideals of $R$, one at least of which is regular, then the ideals $\ga R_{\gp}, \gb R_{\gp}$ are comparable.

\begin{Thm}\cite[Theorem 13]{Gr}\label{regular-tot-ord-prop}
	A ring $R$ is a Pr\"ufer ring if and only if  the pair $(R,\gm)$ has the regular total order property, for every maximal ideal of $R$.
\end{Thm}
Recall that, if $R$ is any ring and $M$ is a $R$-module, the \emph{Nagata idealization of $M$} is the ring $R(+)M:=R\times M$, where addition is defined componentwise and the product is given by
$$
(r,m)(s,n):=(rs,rn+sm),\qquad \mbox{ for all }(r,m),(s,n)\in R(+)M.
$$
For further details, see \cite{Nag} or \cite[Section 25]{Huc}. 
\section{Pr\"ufer extensions}
For the developments about Pr\"ufer rings that will be presented in this article, a useful tool is the general approach to invertibility given by Knebush and Zhang in \cite{KZ}. We will now recall some terminology regarding multiplicative ideal theory of general extensions of rings. 

Let $B$ be any ring and let $A$ be a subring of $B$. 
According to \cite{KZ}, an $A$-submodule $\ga$ of $B$ is {\em $B$-regular} if $\ga B=B$.  $\ga$ is called {\em $B$-invertible} if there exists an $A$-submodule $\gb$ of $B$ such that $\ga \gb=A$. Such an $A$-module $\gb$, if it exists, is uniquely determined, that is, $\gb=[A:_B \ga],$ and it is called the {\em $B$-inverse of $\ga$}.

As it is easily seen, every $B$-invertible ideal is finitely generated and $B$-regular (see \cite[Chapter~II, Remarks~1.10]{KZ}). Furthermore, it is clear that if $B=\Tot(A)$, ``$B$-regular" (resp., ``$B$-invertible") means ``regular" (resp., ``invertible").

 With this notions in mind, it is possible to generalize the concept of Pr\"ufer ring as follows.

\begin{definition}{(\cite[Chapter II, Theorem 2.1]{KZ})}
Let $A\subseteq B$ be a ring extension. We say that {\em $A$ is a Pr\"ufer subring of $B$} if the inclusion $A\hookrightarrow B$ is a flat epimorphism (in the category of rings) and every finitely generated $B$-regular ideal of $A$ is $B$-invertible. If this is the case,  we can also say that {\em $A\subseteq B$ is a Pr\"ufer extension}.
\end{definition}

Obviously, $A\subseteq \Tot(A)$ is a Pr\"ufer extensions if and only if $A$ is a Pr\"ufer ring. In the next proposition, we collect some results on Pr\"ufer extensions that can be found in \cite[pp.~50--52]{KZ}. We will largely make use of them in the next sections.

\begin{Prop}\label{permanence} The following properties hold.
\begin{enumerate}
\item
If $R \subseteq T$ is a Pr\"ufer extension, then for every $T$-overring $S$ of $R$, $R$ is Pr\"ufer in $S$ and $S$ is Pr\"ufer in $T$.
\item
If $R\subseteq S$ and $S\subseteq T$ are Pr\"ufer extensions, then $R\subseteq T$ is a Pr\"ufer extension.
\item
Let $R\subseteq T$ be a ring extension and $\gi$ an ideal of $T$ contained in $R$. Then $R$ is Pr\"ufer in $T$ if and only if $R/\gi$ is Pr\"ufer in $T/\gi$.
\end{enumerate}
\end{Prop}

We start by giving the generalization to arbitrary ring extensions of a well-known fact regarding invertibility of ideals of local domains. We provide the proof for the convenience of the reader. 

\begin{Prop}\label{inv-local}
	Let $A\subseteq B$ be a ring extension where $A$ is a local ring, and let $\f f:=(f_1,\ldots,f_n)A$ be a $B$-invertible $A$-submodule of $B$. Then $\f f=f_iA$, for some $1\leq i\leq n$. 
\end{Prop}
\begin{proof}
By assumption, there are elements $z_1,\ldots, z_n\in [A:_B\f f]$ such that $\sum_{i=1}^nf_iz_i=1$. Note that  $f_iz_i\in A$ for each $1\leq i\leq n$ and thus, since $A$ is local, there is some index $\overline i$ such that $f_{\overline{ i}}z_{\overline{ i}}$ is a unit in $A$. It immediately follows that $\f f=f_{\overline{ i}}A$
\end{proof}

 We now exhibit two results of independent interest, regarding Pr\"ufer extensions, that  generalize those for local Pr\"ufer rings proved in \cite{Boy}.


\begin{Prop}\label{notinvprime}
Let $A$ be a local ring and let $A\subseteq B$ be a Pr\"ufer extension. Then the set of all elements of $A$ that are not invertible in $B$ is a prime ideal of $A$.
\end{Prop}

\begin{proof}
Set $\gp:=\{a \in A \mid a \notin \Cal U(B)\}$. Then $A \setminus \gp$ is a saturated multiplicatively closed subset of $A$. By Zorn's Lemma, it is possible to find a prime ideal $\gq \in \Spec(A)$ maximal with respect to the property of being contained in $\gp$. If $\gq \subsetneq \gp$, take $p \in \gp \setminus \gq$ and consider the ideal $\ga:=pA+\gq$ of $A$. Then there exist $a \in A$ and $q \in \gq$ such that $pa+q \in \Cal U(B)$, which implies that the two-generated ideal $(p,q)A$ of $A$ is $B$-regular and hence $B$-invertible.
Since $A$ is a local ring, we have either $(p,q)A=pA$ or $(p,q)A=qA$, by Proposition \ref{inv-local}. In the first case we have that $B=(p,q)B=pB$, which is not possible since $p \in \gp$. In the second case we have $qA\subset \gq$, which implies that $p \in \gq$, a contradiction. It follows that $\gp=\gq$, hence $\gp$ is a prime ideal of $A$.
\end{proof}

\begin{corollary}\cite[Lemma 3.5]{Boy}
If $A$ is a local Pr\"ufer ring, then the set $Z(A)$ of zero-divisors is a prime ideal.
\end{corollary}

\begin{Prop}
Let $A\subseteq B$ be a Pr\"ufer extension. Assume that $A$ is a local ring and that $B=A_S$ for some multiplicative closed subset of $A$. If $R$ is a subring of $B$ containing $A$, then $R$ is a local Pr\"ufer ring and $R=A_\gp$ for some prime ideal $\gp$ of $A$.
\end{Prop}


\begin{proof} We can assume, without loss of generality, that the multiplicatively closed subset $S$ is saturated.
Set $\gp:=\{a \in A \mid a \notin \Cal U(R)\}$. Since $A\subseteq R$  is a Pr\"ufer extension, Proposition \ref{notinvprime} ensures that $\gp$ is a prime ideal of $A$.

It is clear that $A_\gp \subseteq R$. Let $r \in R$ and write $r=a/s$ for some $a \in A$ and $s \in S$. Since $A$ is a local ring, we have $(a,s)A=aA$ or $(a,s)A=sA$, by Proposition \ref{inv-local}. If $(a,s)A=sA$, then $a=sx$ for some $x \in A$ and $r=x/1 \in A_\gp$. If $(a,s)A=aA$, then we can write $s=a y$ for some $y \in A$ and since $S$ is saturated, both $a$ and $y$ turn out to be in $S$. We have $r=a/s=a/ay=1/y$, which implies $y \notin \gp$. Therefore, $t\in A_\gp$.
\end{proof}

\begin{corollary}\cite[Lemma 3.6]{Boy}
If $A$ is a local ring with Pr\"ufer condition $(n)$ and if $R$ is an overring of $A$, then $R$ is a local ring with Pr\"ufer condition $(n)$. Moreover, $R = A_\gp$ for some prime ideal $\gp$ of $A$.
\end{corollary}

\section{Pr\"ufer conditions in distinguished pullbacks}

\begin{Thm}\label{pullback}
Let $\pi:B\rightarrow T$ be a surjective ring homomorphism, where $T$ is an overring of some ring $R$. Assume that $\ker(\pi)$ is a regular ideal of $B$. Set $A:=\pi^{-1}(R)$.
\begin{enumerate}
\item
$A$ is a Pr\"ufer ring if and only if both $B$ and $R$ are Pr\"ufer rings;
\item
$A$ is a Gaussian (resp. arithmetical) ring if and only if both $B$ and $R$ are Gaussian (resp. arithmetical) rings;
\item
If both $B$ and $R$ are rings with weak global dimension $\leq 1$ (resp. semi-hereditary rings), then so is $A$.
\end{enumerate}
\end{Thm}

\begin{proof}
The ring $A$ is defined by the following pullback diagram:

\begin{equation*}
\xymatrix{
A \ar[r] \ar@{^{(}->}[d] &  R \ar@{^{(}->}[d] \\
B \ar[r]^{\pi} & T
}
\end{equation*}

First assume that $A$ is a Pr\"ufer ring. Since $\ker(\pi)$ is contained in the conductor of the ring extension $A\subseteq B$, $B$ is an overring of $A$ and hence $B$ is a Pr\"ufer ring. Moreover, since $\ker(\pi)$ is a regular ideal, also $R$ is a Pr\"ufer ring.
In the same way, we can prove the ``only-if part'' in $(2)$.

Now assume that both $R$ and $B$ are Pr\"ufer rings. Since $T$ is an overring of $R$, it follows that $R\subseteq T$ is a Pr\"ufer extension (cf. \cite[Chapter~I, Corollary~5.3]{KZ}). By \cite[Chapter~I, Proposition~5.8]{KZ}, $A\subseteq B$ is a Pr\"ufer extension. Applying \cite[Chapter~I, Theorem~5.6]{KZ} to the Pr\"ufer extensions $A \subseteq B$ and $B\subseteq \Tot(A)=\Tot(B)$, we conclude that $A$ is a Pr\"ufer ring.

Finally, assume that both $B$ and $R$ are rings with Pr\"ufer condition $(n)$, for $n=1,2,3,4$. Then $B$ and $R$ are Pr\"ufer rings and therefore so is $A$. Moreover, $\Tot(B)=\Tot(A)$ has Pr\"ufer condition $(n)$. By \cite[Theorem~5.7]{Ba-Gl}, $A$ has Pr\"ufer condition $(n)$.

\end{proof}

{
We now want to list several applications of Theorem~\ref{pullback}. We provide some new examples and constructions of Pr\"ufer rings and we show how this theorem generalizes other classical results.
}
\begin{corollary}
 We preserve the notation and the assumptions of Theorem \ref{pullback}. 
Then $A$ is a local ring with Pr\"ufer condition $(n)$ if and only if $B$ is a local ring with Pr\"ufer condition $(n)$, $R$ is a local Pr\"ufer ring and $T$ is an overring of $R$ (see \cite[Theorem~4.1]{Boy} for the case $\ker(\pi)=[A:_A B]$).
\end{corollary}

\begin{proof}  Assume that $A$ is a local ring with  Pr\"ufer condition $(n)$ (for $1\leq n \leq 5$). We can argue in the same way of \cite[Theorem~4.1]{Boy}. By \cite[Lemma~3.6]{Boy}, $B$ is a local ring with  Pr\"ufer condition $(n)$. Since $\ker(\pi)$ is a regular ideal and $A$ is a Pr\"ufer ring, $R$ is a Pr\"ufer ring (\cite[Proposition~4.4]{FCA1}. Finally, it is immediate that $T$ is an overring of $R$. Indeed, since $B$ is a localization of $A$, say $B=S^{-1} A$ (see \cite[Lemma~3.6]{Boy}), $T=\bar{S}^{-1}R$ and there are no zero-divisors of $R$ in $\bar{S}$ because $R\subseteq T$.

Conversely, if both $B$ and $R$ are local rings, then so is $A$ (see \cite[Corollary~2.4 (3)]{DFF2}).
Since $B$ and $R$ are Pr\"ufer rings, Theorem~\ref{pullback} implies that $A$ is a Pr\"ufer ring. Moreover, since $B$ has Pr\"ufer condition $(n)$, $\Tot(A)$ has Pr\"ufer condition $(n)$. By \cite[Theorem~5.7]{Ba-Gl}, $A$ has Pr\"ufer condition $(n)$.
\end{proof}

The previous result was given by Boynton for the case $\ker(\pi)=[A:_A B]$ (\cite[Theorem~4.1]{Boy}). Nevertheless, the proof of that result do not really need $\ker(\pi)$ to be the conductor of $A$ in $B$, but just a common regular ideal of the two rings. 
Thus, Theorem~\ref{pullback} can be seen as a global version of \cite[Theorem~4.1]{Boy}. Furthermore, it is worth pointing out that in the global version of \cite[Theorem~4.1]{Boy} given in \cite[Theorem~4.2]{Boy} certain conditions on some localizations are involved, while in our theorem we require for $T$ to be an overring of $R$. This hypothesis sounds indeed quite natural and allows several interesting applications of our result (this condition will be further discussed below, see from the comment before Example~\ref{Example 3.7} to Corollary~\ref{Corollary_Fontana}).

\begin{corollary}\label{R+XTX}
	Let $R$ be a ring with total quotient ring $T=\Tot(R)$. Then $R+XT[X]$ is a semihereditary ring if and only if $R$ is a semihereditary ring and $T$ is absolutely flat.
\end{corollary}

\begin{proof}
Recall that $T$ is absolutely flat if and only if $T[X]$ is semihereditary (cf. \cite{McC} and \cite{Cam}).
	
Assume that $R$ is a semihereditary ring and that $T$ is absolutely flat. Then $T[X]$ is semihereditary. Theorem~\ref{pullback} implies that $R+XT[X]$ is a semihereditary ring.
	
On the other hand, if $R+XT[X]$ is semihereditary, then so is $T[X]$, because the conductor of the ring extension $R+XT[X] \subseteq T[X]$ is clearly regular. In particular, $T$ is absolutely flat. Moreover $R$ is a Pr\"ufer ring by \cite[Proposition~4.4]{FCA1}. By \cite[Theorem~5.7]{Ba-Gl}, $R$ is a semihereditary ring.
\end{proof}

\begin{remark}
In \cite{Cam} it is proved that a ring $T$ is absolutely flat if and only if $T[X]$ is semihereditary if and only if $T[X]$ is arithmetical. In view of this result, it is clear that Corollary~\ref{R+XTX} can be restated replacing ``semihereditary'' with ``arithmetical'' or ``$\operatorname{w.gl.dim}(R)\leq 1$'', so that, $R+XT[X]$ is an arithmetical ring (resp. of $\operatorname{w.gl.dim}(R)\leq 1$) if and only if $R$ is an arithmetical ring (resp. of $\operatorname{w.gl.dim}(R)\leq 1$) and $T$ is absolutely flat.
\end{remark}

The next application of Theorem~\ref{pullback} concerns the so-called {\em Pr\"ufer Manis rings}. Here, we present Pr\"ufer Manis rings in a very short way tailored for our purposes.  We refer to \cite{Man}, \cite{BoiLar1} and \cite[Chapter~2]{Huc} for an exaustive description of this topic. A {\em Manis pair} $(A,\gp)$ is a pair where $A$ is a ring, $\gp$ is a prime ideal of $A$ and for every $x \in \Tot(A)\setminus A$, there exists $y \in \gp$ such that $xy \in A\setminus \gp$. Given a ring $A$ and a prime ideal $\f m$ of $A$, $A$ is called a \emph{Pr\"ufer Manis ring} if the following equivalent conditions hold (see \cite[Theorem 2.3]{BoiLar1}):
\begin{enumerate}
	\item $(A,\f m)$ is a Manis pair and $A$ is a Pr\"ufer ring.
	\item $A$ is a Pr\"ufer ring and $\f m$ is the unique regular maximal ideal of $A$.
	\item $(A,\f m)$ is a Manis pair and $\f m$ is the unique regular maximal ideal of $A$. 
\end{enumerate}

\begin{Prop}
	Let $B$ be a Pr\"ufer Manis ring and let $V$ be a valuation domain with quotient field $B/\gm$, where $\gm$ denotes the unique regular maximal ideal of $B$. Consider the canonical projection $\pi: B\rightarrow B/\gm$. Then $\pi^{-1}(V)$ is a Pr\"ufer Manis ring.
\end{Prop}

\begin{proof}
	Set $A:=\pi^{-1}(V)$ and $k:=B/\gm$. We have the following pullback diagram:
	\begin{equation*}
	\xymatrix{
		A \ar[r] \ar@{^{(}->}[d] &  V \ar@{^{(}->}[d] \\
		B \ar[r]^{\pi} & B/\gm
	}
	\end{equation*}
	By Theorem \ref{pullback}, $A$ is a Pr\"ufer ring. Let $\gp:=\pi^{-1}(\gm_V)$ be the contraction of the maximal ideal $\gm_V$ of $V$. We prove that $(A, \gp)$ is a Manis pair. It suffices to show that for every $x \in \Tot(A)\setminus A$ there exists $y \in \gp$ such that $xy \in A \setminus \gp$. Let $x \in \Tot(A)\setminus A$. We distinguish two cases:
	
	Case $(1)$: $x \in B$. Then $\pi(x) \in k \setminus V$ and so $1/\pi(x)\in \gm_V$. It means that there exists $y \in \gp$ such that $\pi(y)\pi(x)=1$. In particular, $xy \in A \setminus \gp$.
	
	Case $(2)$: $x \notin B$. Since $(B,\gm)$ is a Manis pair, there exists $y \in \gm$ such that $xy \in B \setminus \gm$. Therefore $\pi(xy) \in k^*$ and so there exists $z \in B$ such that $\pi(xyz)=1$. In particular, $yz \in \gm=\ker(\pi)\subseteq \gp$ and $xyz \in A \setminus \gp$.
\end{proof}

\begin{corollary}\cite[Theorem~2]{BoiLar} A Pr\"ufer ring is the homomorphic image of a Pr\"ufer domain
if and only if its total quotient ring is the homomorphic image of a Pr\"ufer domain.
\end{corollary}

In their article, Boisen and Larsen provide a very elegant proof of this result. Here, we deduce the ``if-part'' from Theorem~\ref{pullback}, and we refer to \cite{BoiLar} for the proof of the other implication.

\begin{proof}
Let $R$ be a Pr\"ufer ring with total quotient ring $T$. Assume that $T$ is the homomorphic image of a Pr\"ufer domain $B$ under the ring morphism $f$. Then $\ker(f)$ is a regular ideal of $B$ and Theorem \ref{pullback} implies that $f^{-1}(R)$ is a Pr\"ufer domain.
\end{proof}

At this point, the role of the assumption that $T$ is an overring of $R$ in Theorem \ref{pullback} should be underlined. The following two examples (provided both for the case of integral domains and the case of rings with zero-divisors) show that this assumption can not be dropped in the ``if-parts'' of Theorem~\ref{pullback}.

\begin{example}\label{Example 3.7} Let $X,Y$ be two indeterminates over a field $k$. Consider the following pullback diagram:
\begin{equation*}
\xymatrix{
k+Yk(X)[Y]_{(Y)} \ar[r] \ar@{^{(}->}[d] &  k \ar@{^{(}->}[d] \\
k(X)[Y]_{(Y)} \ar[r]^{\pi} & k(X)
}
\end{equation*}
Both $k$ and $k(X)[Y]_{(Y)}$ are (local) Pr\"ufer rings, $\ker(\pi)$ is clearly a regular ideal of $k(X)[Y]_{(Y)}$, but $A:=k+Yk(X)[Y]_{(Y)}$ is not a Pr\"ufer ring. Indeed, $A$ is a local domain, by Remark \ref{fiber-product}(3), but it is not a valuation domain, since $X,X^{-1}$ are in the quotient field of $A$ but none of them belongs to $A$. 
\end{example}

\begin{example} Let $X$ be an indeterminate over $\Q$. Consider the following pullback diagram:
\begin{equation*}
\xymatrix{
\Z+X^2\Q[X] \ar[r] \ar@{^{(}->}[d] &  \Z \ar@{^{(}->}[d]\\
\Z+X\Q[X] \ar[r] & \frac{\Z+X\Q[X]}{(X^2)}
}
\end{equation*}

Both $\Z+X\Q[X]$ and $\Z$ are Pr\"ufer rings, the kernel of the bottom morphism is regular, but $\Z+X^2\Q[X]$ is not a Pr\"ufer ring (it can be easly seen applying, for instance, \cite[Theorem~1.3]{HT}). Notice that condition $(1)$ of Proposition \ref{Prop} gives an easy way to see that $\Z\subseteq \frac{\Z+X\Q[X]}{(X^2)}$ is not a Pr\"ufer extension. Indeed, for every element $\bar{f} \in Z(\frac{\Z+X\Q[X]}{(X^2)})=\frac{X\Q[X]}{(X^2)}$, we have $\Ann_\Z(\bar{f})=0$.
\end{example}

On the other hand, in the notation of Theorem~\ref{pullback}, assume that $A$ is a Pr\"ufer ring and that $\ker(\pi)$ is a regular ideal of $B$. By \cite[Chapter~1, Corollary~5.3]{KZ}, we get that $A\subseteq B$ is a Pr\"ufer extension and by \cite[Chapter~1, Proposition~5.8]{KZ}, $R$ is Pr\"ufer in $T$. As far as the total quotient rings of $R$ and $T$ is concerned, we have the following proposition.

\begin{Prop}\label{Prop}
Let $A\subseteq B$ be a Pr\"ufer extension. Then:
\begin{enumerate}
\item
$\Ann_A(x)\neq 0$ for every $x \in Z(B)$;
\item
$\Tot(A)\subseteq \Tot(B)$;
\item
if $B$ is a Pr\"ufer ring, then $\Tot(A)$ is Pr\"ufer in $\Tot(B)$;
\item
If $\Tot(A)$ is absolutely flat and $B$ is a Pr\"ufer ring, then $\Tot(A)$ $=\Tot(B)$.
\item
If $A$ is a domain and $B$ is a Pr\"ufer ring, then $\Tot(A)=\Tot(B)$. In particular, $B$ must be a domain.
\end{enumerate}
\end{Prop}

\begin{proof}
$(1)$ Let $x \in Z(B)$ and pick an element $z \in Z(B)$ such that $zx=0$. By \cite[Chapter~1, Theorem~5.2 (8)]{KZ}, we can write $1=a_1+a_2z$ for suitables $a_1,a_2 \in (A:z)$. It follows that $x(1-a_1)=0$. If $1-a_1=0$, then $z \in A$, otherwise $1-a_1 \in \Ann_A(x)$.

$(2)$ It suffices to prove that $\Reg(A)\subseteq \Reg(B)$. Let $r \in \Reg(A)$ and assume that there exists $x \in Z(B)$ such that $rx=0$. By \cite[Chapter~1, Theorem~5.2 (8)]{KZ}, there exist $a_1,a_2 \in (A\colon x)$ such that $1=a_1+ x a_2$, so that $r=r a_1$. Since $r$ is a regular element of $A$, it follows that $a_1=1$, that is $x \in A$, a contradiction.

$(3)$ By assumption, both $A\subseteq B$ and $B\subseteq \Tot(B)$ are Pr\"ufer extensions. Applying \cite[Chapter1, Theorem~5.6]{KZ}, we get that $A$ is Pr\"ufer in $\Tot(B)$. From $(2)$ and \cite[Chapter1, Corollary~5.3]{KZ}, $\Tot(A)\subseteq \Tot(B)$ is a Pr\"ufer extension.

$(4)$ By $(3)$, $\Tot(A)\subseteq \Tot(B)$ is a Pr\"ufer extension, so that the inclusion $\Tot(A)\hookrightarrow \Tot(B)$ is a (flat) epimorphism. But a non-trivial epimorphism from an absolutely flat ring into an arbitrary ring is necessarily an isomorphism.

$(5)$ Immediately follows by $(4)$.
\end{proof}

In view of these results, a question arises naturally:

\begin{question}
In the notation of Theorem~\ref{pullback}, assume that $\ker(\pi)$ is a regular ideal of $B$ and that $A$ is a Pr\"ufer ring. In which cases is it possible to deduce that $T$ is an overring of $R$?
\end{question}

For instance, if $A$ is a local Pr\"ufer ring, then $B$ is a localization of $A$ (cf. \cite[Lemma 3.6]{Boy}) and a straightforward argument shows that $T$ is an overring of $R$.  Moreover, as we have seen in Proposition~\ref{Prop}, we can deduce that $T$ is an overring of $R$ also if $\Tot(R)$ is an absolutely flat ring (so, in particular, if $R$ is a domain). In this latter case, all five Pr\"ufer conditions are equivalent for $R$ (cf. \cite[Theorem~5.7]{Ba-Gl}), hence Theorem~\ref{pullback} can be restated as follows.

\begin{Prop}
Let $\pi:B\rightarrow T$ be a surjective ring homomorphism, and let $R$ be a subring of $T$. Assume that $\Tot(R)$ is an absolutely flat ring and that $\ker(\pi)$ is a regular ideal of $B$. Set $A:=\pi^{-1}(R)$. Then, for $n=1,\dots,5$, $A$ has Pr\"ufer condition $(n)$ if and only if both $B$ and $R$ have the same Pr\"ufer condition $(n)$ and $T$ is an overring of $R$.
\end{Prop}

In particular, we can recover a result of Houston and Taylor in the case of integral domains. We preserve the notations of the original paper.

\begin{corollary}\cite[Theorem~1.3]{HT}
Let $T$ be a domain and let $\gi$ be an ideal of $T$. Let $D$ be a domain contained in $E:=T/\gi$ and let $\pi:T \rightarrow E$ denote the canonical projecton. Then $R:=\pi^{-1}(D)$ is a Pr\"ufer ring if and only if both $D$ and $T$ are Pr\"ufer rings, $\gi$ is a prime ideal of $T$ and $D,E$ have the same field of fractions.
\end{corollary}

As a further specialization, we have

\begin{corollary}\label{Corollary_Fontana}(see \cite[Theorem~2.4~(3)]{FO1} for the local case) Let $D$ be a domain and let $\gm$ be a maximal ideal of $D$. Let $E$ be a domain having $D/\gm$ as field of quotients. Set $D_1:=\pi^{-1}(E)$, where $\pi:D\rightarrow D/\gm$ is the canonical projection. Then $D_1$ is a Pr\"ufer domain if and only if both $D$ and $E$ are Pr\"ufer domains.
\end{corollary}

We conclude this section with another result concerning the Pr\"ufer property for a distinguished class of fiber products.

\begin{Prop}\label{fib-reg}
Let $f:A\rightarrow C$ and $g:B\rightarrow C$ be ring morphisms, and assume that $\ker(f),\ker(g)$ are regular ideals of $A$, $B$, respectively.  Then the following conditions are equivalent:
\begin{enumerate}
\item
$f\times_C g$ is a Pr\"ufer ring
\item
$A$ and $B$ are Pr\"ufer rings and $f\times_C g = A\times B$.
\end{enumerate}
\end{Prop}
\begin{proof}
$(2) \Rightarrow (1)$ The product of Pr\"ufer rings is a Pr\"ufer ring, by \cite[Proposition 3]{gi-hu}.

$(1)\Rightarrow(2)$. First, it is easy to verify that the conductor of the ring extension $D:= f\times_C g\subseteq A\times B$ is $\f c:= \ker(f)\times \ker(g)$. Assume, by contradiction, that $D\subsetneq A\times B$, that is, $\f c$ is a proper ideal of $D$, and let $\f m$ be a maximal ideal of $D$ containing $\f c$. If $S_A$ (resp., $S_B$) is the image of $D\setminus \f m$ under the natural morphism $D\to A$ (resp., $D\to B$), $S_C:=f(S_A)=g(S_B)$ and $\widetilde{f}:A_{S_A}\to C_{S_C}$, $\widetilde{g}:B_{S_B}\to C_{S_C}$ are the morphisms induced by $f,g$, respectively, on the localizations, then $D_{\f m}$ is canonically isomorphic to $R:=\widetilde{f}\times_{C_{S_C}}\widetilde{ g}$, by \cite[Proposition 1.9]{FO1}. By assumptions there are regular elements $a\in \ker(f), b\in \ker(g)$ and, in particular, their images  $a/1\in A_{S_A}$ and $b/1 \in B_{S_B}$ are regular. It follows that the element $(a/1,b/1)\in R$ is regular. According, to \cite[Theorem 13]{Gr}, the ideals $(a/1,b/1)R, (a/1,0)R$ of $R$ are comparable. If there is $(\rho, \sigma)\in R$ such that $(a/1,b/1)=(a/1,0)(\rho,\sigma)$, we have $b/1=0$ in $B_{S_B}$, in particular, against the fact that $b$ is regular in $B$. On the other hand, if there is an element $(\eta,\zeta)\in R$ such that $(a/1,0)=(a/1,b/1)(\eta, \zeta)$. Keeping in mind that $a/1,b/1$ are regular in $A_{S_A}, B_{S_B}$, respectively, it follows that $\eta=1,\zeta=0$ and, since $(\eta,\zeta)\in R$, 
$$
1=\widetilde{ f}(\eta)=\widetilde{ g}(\zeta)=0.
$$
This is a contradiction because since $\f c\subseteq \f m$ we easily infer that $0\notin S_C$. This proves that $f\times_Cg=A\times B$ and, by (1), it is Pr\"ufer. The conclusion follows from \cite[Proposition 3]{gi-hu}.
\end{proof}

\begin{corollary}
Preserve the notation and the assumptions of Proposition \ref{fib-reg}. Then, for $n=1,\dots, 5$ the following conditions are equivalent.
\begin{enumerate}
	\item $f\times_Cg$ has Pr\"ufer condition $(n)$. 
	\item $A,B$ have Pr\"ufer condition $(n)$ and $f\times_Cg=A\times B$. 
\end{enumerate}
\end{corollary}

\begin{proof}
It suffices to notice that for rings $A$ and $B$, $A\times B$ has Pr\"ufer condition $(n)$ if and only if both $A$ and $B$ have the same Pr\"ufer condition $(n)$.
\end{proof}

\begin{remark}
Let $R\subseteq T$ be a ring extension with conductor $\Gamma$. Then $R$ is isomorphic to the fiber product $\pi\times_{T/\Gamma} \iota$, where $\pi:T\rightarrow T/\Gamma$ is the canonical projection and $\iota: R/\Gamma \hookrightarrow T/\Gamma$ is the canonical embedding. So, $R$ can be viewed as a subring of $(R/\Gamma) \times T$ and the conductor of the ring extension $R\cong \pi\times_{T/\Gamma} \iota \subseteq (R/\Gamma) \times T$ is $0\times \Gamma$, which is never regular in $(R/\Gamma) \times T$.
\end{remark}

\section{Homomorphic images and  Pr\"ufer rings}
\subsection{Regular morphisms}
According to \cite[Proposition~4.4]{FCA1}, if $A$ is a Pr\"ufer ring and $\ga$ is an ideal of $A$, then $A/\ga$ is a Pr\"ufer ring whenever $\ga$ is regular. As we have seen, this fact is crucial in the proof of Theorem~\ref{pullback}. The first goal of this section is to present a notion that allows to consider homomorphic images of Pr\"ufer rings that are still Pr\"ufer, without taking regular ideals. Our definition mimic that of local morphisms (recall that a ring morphism $f:A\rightarrow B$ is said to be {\em local} if $f^{-1}(\Cal U(B))\subseteq \Cal U(A)$, that is, if for every $a \in A$, $f(a)$ is invertible in $B$ if and only if $a$ is invertible in $A$ \cite{CD}).

\begin{definition}
Let $f:A\rightarrow B$ be a ring morphism. We say that $f$ is a {\em regular morphism} if $f^{-1}(\Reg(B))\subseteq \Reg(A)$. We say that $B$ is a {\em regular homomorphic image} of $A$ if there exists a surjective regular morphism $f:A \rightarrow B$.
\end{definition}

Let us start with the following lemma. Its proof is straightforward and thus it is left to the reader. 

\begin{Lemma}\label{Lemma1} Let $f:A\rightarrow B$ be a ring morphism. Then the following properties hold. 

\begin{enumerate}
\item
$B$ is a regular homomorphic image of $A$ via $f$ if and only if $f$ is surjective and $Z(A)\subseteq f^{-1}(Z(B))$.
\item
If $f$ is a local morphism and $B$ is a total ring of quotients, then $f$ is a regular morphism.
\item
If $f$ is a regular morphism and $A$ is a total ring of quotients, then $f$ is a local morphism.
\item
If $f$ is surjective, then $f$ is a regular morphism if and only if for every $a \in A$, $(\Ker(f):a)\subseteq \Ker(f)$ implies $a \in \Reg(A)$.
\item
If $Z(A)$ is contained in a proper ideal $\gi$ of $A$, then $A\rightarrow A/\gi$ is a regular morphism.
\item
If $A$ is a ring in which every zero-divisor is nilpotent, then $A\rightarrow A/\gp$ is a regular morphism for every $\gp \in \Spec(A)$.
\end{enumerate}
\end{Lemma}
%
\begin{Thm}\label{regprufer}
Every regular homomorphic image of a Pr\"ufer ring is a Pr\"ufer ring.
\end{Thm}

\begin{proof}
Let $f:A\rightarrow B$ be a regular surjective morphism and assume that $A$ is a Pr\"ufer ring.
Let $\gb:=(b_1,\dots,b_n)$ be a finitely generated regular ideal of $B$. For every $i=1,\dots,n$, we can write $b_i=f(a_i)$ for some $a_i \in A$. Since $f$ is a regular morphism, the finitely generated ideal $\ga:=(a_1,\dots,a_n)$ of $A$ is a regular ideal and, since $A$ is a Pr\"ufer ring, $\ga$ is locally principal, being it invertible. If $\f m$ is a maximal ideal of $B$, $\f n:=f^{-1}(\f m)$ and $f':A_{\f n}\to B_{\f m}$ is the canonical surjective morphism induced by $f$, we have $\f bB_{\f m}=f'(\f aA_{\f n})$ and thus it is invertible. 
%
%
%
%
%
\end{proof}

\begin{corollary}
Let $A$ be a local Pr\"ufer ring. Then $A/Z(A)$ is a Pr\"ufer domain.
\end{corollary}

\begin{proof}
By \cite[Lemma~3.5]{Boy}, $Z(A)$ is a prime ideal of $A$. The morphism $A\rightarrow A/Z(A)$ is regular by Lemma \ref{Lemma1} $(5)$, and so $A/Z(A)$ is Pr\"ufer by Theorem \ref{regprufer}.
\end{proof}

\begin{example}
Let $A$ be a ring and let $M$ be an $A$-module such that for every element $x \in M\setminus\{0\}$, $\Ann(x)\subseteq Z(A)$. Consider the idealization $A(+)M$ of $M$ in $A$. Then, it is easily seen that the canonical map $A(+)M \rightarrow A$ is a regular morphism. In particular, if $A(+)M$ is a Pr\"ufer ring, then so is $A$.
\end{example}

Combining Theorems~\ref{pullback} and~\ref{regprufer} we can derive the following result concerning Pr\"ufer rings of the type $k+\f m$. 

\begin{corollary}\cite[Theorem~2.1]{BakMah} Let $(T,\gm)$ be a local ring of the form $T=k+\gm$, for some field $k$. Take a subring $D$ of $k$ such that $\mbox{Q}(D)=k$ and set $R:=D+\gm$. Then $R$ has Pr\"ufer condition $(n)$ if and only if $T$ and $D$ have the same Pr\"ufer condition $(n)$.
\end{corollary}

\begin{proof} We have the following pullback diagram:
\begin{equation*}
\xymatrix{
R \ar[r]^{\pi_0} \ar@{^{(}->}[d] &  D \ar@{^{(}->}[d]\\
T \ar[r]^{\pi} & k
}
\end{equation*}
where $\pi$ and $\pi_0$ are the canonical projections. If $\gm$ is a regular ideal of $T$, then it suffices to apply Theorem~\ref{pullback}.

So, assume that $\gm$ consists only of zero divisors of $T$. It is immediate that $T$ is a total ring of quotients, $\gm=Z(T)$ and $\Tot(R)=T$. Moreover, $\pi_0$ is a regular morphism, because $\pi_0^{-1}(D\setminus\{0\})=R\setminus \gm\subseteq \Reg(R)$. So, if $R$ is a Pr\"ufer ring, then so are $T$ and $D$. 

For the other implication, assume that both $D$ and $T$ are Pr\"ufer rings. Then, $D \subseteq k$ is a Pr\"ufer extension, and by \cite[Chapter~I, Proposition~5.8]{KZ}, so is $R\subseteq T$. Since $\Tot(R)=T$, $R$ is a Pr\"ufer ring.

To conclude, it suffices to notice that all five Pr\"ufer conditions coincide on $D$ and that if $R$ is a Pr\"ufer ring, then $R$ has the Pr\"ufer condition $(n)$ if and only if $T=\Tot(R)$ has the same Pr\"ufer condition $(n)$ \cite[Theorem~5.7]{Ba-Gl}.
\end{proof}

\subsection{Pre-Pr\"ufer rings}
According to \cite{BoiShe}, a ring morphism $A\to B$ is \emph{proper} if its kernel is different from $0$ and from $A$. 
Then, the authors define a ring $R$ to be a {\em pre-Pr\"ufer ring} if every proper homomorphic image of $R$ is a Pr\"ufer ring. They show that the prime spectrum of a pre-Pr\"ufer domain forms a tree \cite[Theorem~1.2]{BoiShe} and that in the Noetherian case, pre-Pr\"ufer domains are precisely those of dimension 1 \cite[Corollary 1.3]{BoiShe}. The aim of this short section is to pointing out some generalization of these two results in the case of rings with zero divisors.

Let $\Cal I$ be a family of ideals of a ring $R$. Inspired by the notion of the regular total order property for pairs given by Griffin \cite{Gr}, we say that $\Cal I$ has the {\em regular total order property} if for every pair of ideals $I,J \in \Cal I$, where at least one of them is regular, $I$ and $J$ are comparable. It is clear that if $R$ is a domain, then $\Cal I$ has the regular total order property if and only if $\Cal I$ is a chain. The prime spectrum of a ring $R$ forms a tree if and only if $\Spec(R_\gm)$ is linearly ordered for each maximal ideal $\gm$ of $R$. For pre-Pr\"ufer rings, we have the following result (that generalizes \cite[Theorem 1.2]{BoiShe}).

\begin{theorem}\label{regular-ord-pre-pru}
Let $A$ be a pre-Pr\"ufer ring. Then $\Spec(A_\gm)$ has the regular total order property for every maximal ideal $\gm$ of $A$.
\end{theorem}

\begin{proof}
In view of \cite[Theorem 1.1]{BoiShe}, we can assume that $A$ is a local pre-Pr\"ufer ring. Let $\gp$ and $\gq$ be prime ideals of $A$ with $\gp$ regular. We want to prove that $\gp$ and $\gq$ are comparable. If this is not the case, we can certainly assume that $\f q$ is nonzero. Since $\f p$ is regular, we must  have $0\neq \gp \gq \subseteq \gi:= \gp \cap \gq$ and therefore $A/\gi$ is a local Pr\"ufer ring, as $A$ is local pre-Pr\"ufer. It is straightforward to show that $\gp/\gi\cup \gq/\gi=Z(A/\gi)$. By \cite[Lemma 3.5]{Boy}
the set $Z(A/\gi)$ of zero-divisors of $A/\f i$ is a prime ideal of $A/\gi$. Then,  $\gp/\gi$ and  $\gq/\gi$ must be comparable, and this is a contradiction.
\end{proof}

\begin{corollary} Let $R$ be a pre-Pr\"ufer ring. Then the following statements hold. 
	\begin{enumerate}
		\item Two distinct minimal primes over a regular ideal of $R$ are  comaximal. 
		\item If $R$ is local, then every regular ideal of $R$ has a unique minimal prime. 
	\end{enumerate}
\end{corollary}
\begin{proof}
We only need to prove (1), statement (2) being an immediate consequence. Let $\f i$ be a regular ideal of $R$ and assume that $\f p,\f q\supseteq \f i$ are distinct (regular) minimal prime ideals over $\f i$ which are not comaximal. Let $\f m$ be a maximal ideal of $R$ containing $\f p+\f q$. Then $\f pR_{\f m},\f qR_{\f m}$ are both regular prime ideals of $R_{\f m}$ and they are not comparable, against Theorem  \ref{regular-ord-pre-pru}
\end{proof}
The final result is a slight generalization of one implication of \cite[Corollary 1.3]{BoiShe}. 
\begin{Prop}
Let $R$ be a Noetherian pre-Pr\"ufer ring. Then $\dim(A)\leq  1$.
\end{Prop}
\begin{proof}
Assume $\dim(R)> 1$ and let $\gq\subsetneq \gp\subsetneq \gm$ be a chain of prime ideals. By \cite[Theorem 1.1]{BoiShe}, there is no restriction in assuming that $R$ is local with maximal ideal $\f m$. By \cite[Theorem 144]{kap}, the set $X$ of all prime ideals $\f h$ of $R$ such that $\f q\subsetneq \f h\subsetneq \f m$ is infinite. If $\f q=0$, then Theorem \ref{regular-ord-pre-pru} implies that $X$ is an infinite chain and thus $A$ infinite-dimensional, against the fact that $A$ is local and Noetherian. If $\f q\neq 0$, then $R/\f q$ is a local Dedekind domain (since $R$ is pre-Pr\"ufer) of dimension $\geq 2$, another contradiction. 
\end{proof}

In the integral case Noetherian pre-Pr\"ufer domains are precisely those of dimension 1, that is, a Noetherian domain $D$ is a pre-Pr\"ufer domain if and only if $\dim(D)= 1$ \cite[Corollary~1.3]{BoiShe}. Nevertheless, in rings with zero-divisors, there exist one-dimensional Noetherian rings that are not pre-Pr\"ufer, as the following example shows.

\begin{example}
Let $A$ be a Noetherian, one-dimensional local domain which is not Dedekind. Let $k$ be the residue field of $A$, endowed with its natural structure of $A$-module. Consider the idealization $R:=A(+)k$. Then $\dim(A)=\dim(R)=1$ (cf. \cite[Chapter~VI]{Huc}) and $R$ is a Noetherian ring by \cite[Proposition~2.2]{AW}.  Notice also that $R$ is a total ring of quotients, hence a Pr\"ufer ring (indeed, $R$ is a local ring whose maximal ideal is $\f m (+) k:=\f m \times k$, where $\f m $ is the maximal ideal of $A$; moreover, it is easy to check that $\f m (+) k$ consists of zero-divisors). But $R$ is not a pre-Pr\"ufer ring, since $A\cong R/(0(+)k)$ is not a Pr\"ufer ring.
\end{example}
\textbf{Acknowledgements.} The authors would like to express their gratitude to the referee, whose comments and suggestions helped to improve the presentation of the paper.

\bibliographystyle{amsalpha}

\end{document}